\documentclass[11pt]{amsart}

\usepackage{enumitem}
\usepackage{hyperref}
\usepackage{verbatim}
\usepackage{xcolor}

\newcommand{\ph}{\varphi}
\newcommand{\eps}{\varepsilon}
\newcommand{\ulim}{\varlimsup}
\newcommand{\llim}{\varliminf}
\newcommand{\NN}{\mathbb{N}}
\newcommand{\RR}{\mathbb{R}}

\newcommand{\DDD}{\mathcal{D}}
\newcommand{\CCC}{\mathcal{C}}
\newcommand{\LLL}{\mathcal{L}}
\newcommand{\MMM}{\mathcal{M}}

\newcommand{\GGG}{\mathcal{G}}
\newcommand{\JJJ}{\mathcal{J}}
\newcommand{\FFF}{\mathcal{F}}
\newcommand{\one}{\mathbf{1}}
\newcommand{\dH}{d_{\mathrm{Ham}}}
\newcommand{\htop}{h_{\mathrm{top}}}
\newcommand{\vn}{\mathbf{n}}
\newcommand{\Cp}{C^{\mathcal{P}}}
\newcommand{\Cs}{C^{\mathcal{S}}}

\newtheorem{theorem}{Theorem}[section]
\newtheorem{lemma}[theorem]{Lemma}
\newtheorem{proposition}[theorem]{Proposition}
\newtheorem{corollary}[theorem]{Corollary}

\newtheorem*{thma*}{Theorem}

\theoremstyle{remark}
\newtheorem{remark}[theorem]{Remark}

\numberwithin{equation}{section}

\begin{document}

\title{Positive entropy equilibrium states}
\author{Vaughn Climenhaga}
\author{Van Cyr}
\address{Dept.\ of Mathematics, University of Houston, Houston, TX}
\address{Dept.\ of Mathematics, Bucknell University, Lewisburg, PA}
\email{climenha@math.uh.edu}
\email{van.cyr@bucknell.edu}
\date{\today}

\begin{abstract}
For transitive shifts of finite type, and more generally for shifts with specification, it is well-known that every equilibrium state for a H\"older continuous potential has positive entropy as long as the shift has positive topological entropy.  We give a non-uniform specification condition under which this property continues to hold, and demonstrate that it does not necessarily hold for other non-uniform versions of specification that have been introduced elsewhere.
\end{abstract}
\thanks{The first author is partially supported by NSF grants DMS-1362838 and
DMS-1554794.}
\maketitle

\section{Introduction}

Given a compact metric space $X$, a continuous map $f\colon X\to X$, and a continuous potential function $\ph\colon X\to \RR$, an \emph{equilibrium state} for $(X,f,\ph)$ is an $f$-invariant measure realising the supremum in the variational principle $P(\ph) = \sup_{\mu}(h_\mu(f) + \int \ph\,d\mu)$.
It is often important to know under what conditions an equilibrium state is forced to have positive entropy, or equivalently, for which potentials we have
\begin{equation}\label{eqn:Pbig}
P(\ph) > \sup_\mu \int \ph\,d\mu.
\end{equation}
Following \cite{IRRL12}, a potential satisfying \eqref{eqn:Pbig} will be called \emph{hyperbolic}.

If $(X,\sigma)$ is a transitive subshift of finite type (SFT) with positive topological entropy, then every H\"older potential is hyperbolic.  This also holds for all systems with the specification property \cite[Theorem 6.1]{CFT}.  

The importance of~\eqref{eqn:Pbig} is discussed in~\cite{jB04}; see \cite{DKU90,jB01} for its consequences regarding uniqueness of equilibrium states, and \cite{mR83,gK84,BK90} for its connection to quasi-compactness of the transfer operator, which has implications for the statistical properties of the system.

In \cite{jB04}, Buzzi considers continuous piecewise monotonic interval maps $f$ and shows that if $f$ is topologically transitive and $\ph$ is H\"older continuous in the natural coding via the branch partition, then~\eqref{eqn:Pbig} holds.  Buzzi conjectured that the result remains true without the assumption that the map  $f$ is continuous, but so far this question remains open.





We offer partial progress towards this conjecture by giving a general condition under which every H\"older potential satisfies \eqref{eqn:Pbig}.  Our condition is formulated in terms of the symbolic representation of $f$, and can be thought of as a stronger version of the \emph{almost specification} property \cite{PS07,dT12}.

Given a shift space $X$, we write $\LLL$ for the \emph{language} of $X$ (the set of all finite words appearing in some element of $X$).  A \emph{prefix} of a word $w\in \LLL$ is any word of the form $w_1 \cdots w_k$ for some $k\leq |w|$; similarly, a \emph{suffix} of $w$ is any word of the form $w_k \cdots w_{|w|}$.  We say that a subset $\GGG\subset \LLL$ has \emph{specification} if there is $\tau\in \NN$ such that for every $v,w\in \GGG$ there is $u\in \LLL$ with $|u|\leq \tau$ such that $v'uw' \in \GGG$ whenever $v'$ is a suffix of $v$ and $w'$ is a prefix of $w$ with $v',w'\in \GGG$.  (This includes the case when $v'=v$ and $w'=w$.)  

Given a nondecreasing function $g\colon \NN\to \NN$, the language $\LLL$ is said to be \emph{$g$-Hamming approachable} by $\GGG$ if every sufficiently long $w\in \LLL$ can be transformed into a word in $\GGG$ by changing no more than $g(|w|)$ symbols.

\begin{theorem}\label{thm:hyperbolic}
Let $X$ be a shift space on a finite alphabet with positive topological entropy, and $\LLL$ its language.  If there is a function $g\colon \NN\to \NN$ with $\lim_{n\to\infty}  g(n)/\log(n)= 0$ and a set $\GGG\subset \LLL$ with specification such that $\LLL$ is $g$-Hamming approachable by $\GGG$, then every H\"older continuous potential on $X$ is hyperbolic.
\end{theorem}

An important class of shifts satisfying the conditions of the theorem is given by the $\beta$-shifts, which code the transformations $x\mapsto \beta x \pmod 1$ for $\beta>1$.  In this case $g(n)=1$ for every $n$, and it was already shown in \cite[Proposition 3.1]{CTb} that every H\"older potential is hyperbolic.  The proof there relied strongly on the lexicographic structure of the $\beta$-shifts; in particular it does not apply to their factors.  Our approach here does pass to factors.

\begin{proposition}\label{prop:factors}
Let $X$ be a shift space satisfying the hypotheses of Theorem \ref{thm:hyperbolic}.  Then every subshift factor of $X$ satisfies them as well.
\end{proposition}
\begin{proof}
By the proof of \cite[Lemma 2.12]{CTY}, if $g(n)$ works for $X$, and $\tilde X$ is a subshift factor obtained via an $r$-block code, then $\tilde g(n) = (4r+3) g(n+2r) + 4r$ works for $\tilde X$.
\end{proof}

\begin{corollary}\label{cor:beta-factors}
Let $X$ be any subshift factor of a $\beta$-shift.  Then every H\"older potential on $X$ satisfies \eqref{eqn:Pbig}, and has a unique equilibrium state, which has exponential decay of correlations and the central limit theorem for H\"older observables.
\end{corollary}
\begin{proof}
Theorem \ref{thm:hyperbolic} and Proposition \ref{prop:factors} give \eqref{eqn:Pbig}; for the rest, see \cite[Theorem 1.4, Example 1.5, and \S\S1.7--1.8]{spec-towers}.
\end{proof}

\begin{remark}\label{rem:S-gap}
Another class of shift spaces studied in \cite{CTa,CTY} are the $S$-gap shifts, for which there is no function $g$ as in Theorem \ref{thm:hyperbolic}; the best that can be done in general is $g(n) \approx \sqrt{n}$, see \cite[\S5.1.2]{CTY}.  On the other hand, it was shown in \cite[(5.1)]{CTY} that every H\"older potential for these shifts is hyperbolic.  The corresponding question for their subshift factors remains open.
\end{remark}

\begin{remark}
Another condition that appears in the literature to guarantee hyperbolicity of H\"older potentials is the `local specification' condition of Hofbauer and Keller \cite[Theorem 3]{HK82}, which can be stated as follows.
Given $k\in \NN$, let $\FFF_k$ be the set of $w\in \LLL$ such that for every $v\in \LLL$, there is $u\in \LLL$ with $|u| \leq k$ such that $wuv\in \LLL$.  (Then $\LLL$ has specification iff there is $k$ such that $\FFF_k = \LLL$.)  The `local specification' property from \cite[Theorem 3]{HK82} is equivalent to: for every $x\in X$ and every infinite $J\subset \NN$, there is $k\in \NN$ and an infinite $J' \subset J$ such that $x_1 \cdots x_j \in \FFF_k$ for every $j\in J'$.

Another result for interval maps was given in \cite{LRL14}, which showed that for a class of smooth interval maps with critical points and some non-uniformly expanding properties, \eqref{eqn:Pbig} holds for every H\"older continuous potential (not just those that are H\"older in the natural coding).
\end{remark}

Beyond $\beta$-transformations, it is natural to study the class of interval maps given by $x\mapsto \alpha + \beta x$ for $\alpha\in (0,1)$, $\beta>1$.  The coding spaces for these maps can be represented in terms of a countable graph using the general theory of Hofbauer \cite{fH79}, but it is not clear what mistake function $g$ these shifts admit, and so Buzzi's conjecture remains open for this class.

In light of Remark \ref{rem:S-gap} above on $S$-gap shifts, and other results from \cite{CTY} in which $g(n)/n \to 0$ seems to be the relevant condition, it is natural to ask how sharp the sublogarithmic condition on $g$ is.  In fact, one cannot do much better, as the following family of examples shows.

\begin{theorem}\label{thm:context-free}
Let $f\colon \NN\to \NN$ be nondecreasing, and suppose that there is $n_1\in \NN$ such that $1\leq f(n) \leq n/2$ for all $n\geq n_1$.  Let $G = \{0^a 1^b \mid a,b \geq f(a+b)\}$, and let $X$ be the coded shift generated by $G$.  Then for $\ph = -\one_{[1]}$, the potentials $t\ph$ have $P(t\ph)\geq 0$ for all
 $t\in \RR$, and $t\mapsto P(t\ph)$ is non-increasing.  Writing
\begin{equation}\label{eqn:t0}
t_0 = \inf \{t \mid P(t\ph) = 0\} = \sup \{t \mid P(t\ph)>0\}
\end{equation}
for the first root of Bowen's equation (possibly $+\infty$), the following are true.
\begin{enumerate}[label=\upshape{(\roman{*})}]
\item For the function $g(n) = 2n_1 + 2\max(f(n),n_1)$, $\LLL=\LLL(X)$ is $g$-Hamming approachable by $\GGG = G^*$.
\item Given $t\geq 0$, the potential $t\ph$ is hyperbolic if and only if $t<t_0$.
\item If $0\leq t<t_0$, then there is a unique equilibrium state for $t\ph$, and it has positive entropy.
\item If $t>t_0$, then $\delta_0$ is the unique equilibrium state for $t\ph$.
\item\label{root} $t_0<\infty$ if and only if there exists $\gamma>0$ such that $\sum_{n\in \NN} \gamma^{f(n)} < \infty$.
\end{enumerate}
\end{theorem}

\begin{remark}
The examples in Theorem \ref{thm:context-free} are modifications of the coded shift generated by $G = \{0^n 1^n : n\in \NN\}$, which was studied by Conrad \cite{sC}, who showed that for sufficiently large values of $t$, the potential $t\ph$ has the delta measure $\delta_0$ as its unique equilibrium state, and in particular is not hyperbolic.
\end{remark}


The last statement in Theorem \ref{thm:context-free} allows us to give a class of shifts for which there is a H\"older potential that is not hyperbolic.

\begin{corollary}\label{cor:context-free}
If $\liminf g(n)/\log(n) > 0$, then the conclusion of Theorem \ref{thm:hyperbolic} fails in the following sense: there is a shift $X$ with language $\LLL$ and a collection $\GGG\subset \LLL$ such that $\GGG^* \subset \GGG$ and $\LLL$ is $g$-Hamming approachable by $\GGG$, but there is a locally constant potential function with a delta measure as its unique equilibrium state.
\end{corollary}

\begin{remark}
In fact, Theorem \ref{thm:context-free} shows that hyperbolicity can fail for some error functions $g$ with $\liminf g(n)/\log(n)=0$ and $\limsup g(n)/\log(n)>0$, as long as there is $\gamma>0$ such that $\sum_n \gamma^{g(n)} < \infty$.  This does not cover all functions $g$ with $\liminf=0$ and $\limsup>0$; it would be interesting to know if Theorem \ref{thm:hyperbolic} can be extended to include functions $g$ where $\limsup > 0$ but $\sum_n \gamma^{g(n)} = \infty$ for all $\gamma>0$.
\end{remark}

\subsection*{Acknowledgments} We are grateful to the anonymous referee for pointing out an error in the original version of \S\ref{sec:hamming-approachability}.

\section{Background definitions}

\subsection{Shift spaces}

Given a finite set $A$, let $\sigma\colon A^\NN\to A^\NN$ denote the left shift map.\footnote{
Our results all remain true for two-sided shifts ($\sigma\colon A^\mathbb{Z} \to A^\mathbb{Z}$).}  Equip $A^\NN$ with the product topology; equivalently, define a metric on $A$ by $d(x,y) = 2^{-\min\{n\in \NN \mid x_n \neq y_n\}}$.  A \emph{shift space over the alphabet $A$} is a closed $\sigma$-invariant subset $X \subset A^\NN$.

Write $A^* = \bigcup_{n=0}^\infty A^n$ for the collection of all finite words over $A$.
Given a shift space $X$, the \emph{language} of $X$ is
\[
\LLL = \LLL(X) = \{w\in A^* \mid x_1 \cdots x_n = w \text{ for some } x\in X \text{ and } n\in \NN \}.
\]
Given $\DDD\subset \LLL$, write $\DDD_n = \DDD \cap A^n$ for the set of all words of length $n$ in $\DDD$.  In particular, $\LLL_n$ denotes the set of all words of length $n$ in the language of $X$.  Given $w\in \LLL_n$, let $[w] = \{x\in X \mid x_1 \cdots x_n = w\}$ be the corresponding \emph{cylinder} in $X$.

\subsection{Thermodynamic formalism and equilibrium states}

Let $X$ be a shift space and $\LLL$ its language. Given a continuous function $\ph\colon X\to \RR$, which we call a \emph{potential}, consider for each $w\in \LLL_n$ the quantity
\[
\Phi(w) := \sup_{x\in [w]} S_n \ph(x),
\]
where $S_n\ph(x) = \sum_{k=0}^{n-1} \ph(\sigma^kx)$.  Given $\DDD \subset \LLL$, the $n$th \emph{partition sum} associated to $\DDD$ and $\ph$ is
\[
\Lambda_n(\DDD,\ph) := \sum_{w\in \DDD_n} e^{\Phi(w)}.
\]
The \emph{pressure} of $\DDD$ with respect to $\ph$ is
\[
P(\DDD,\ph) := \ulim_{n\to\infty} \frac 1n \log \Lambda_n(\DDD,\ph).
\]
In the specific case $\ph=0$, this reduces to the \emph{entropy} of $\DDD$:
\[
h(\DDD) := \ulim_{n\to\infty} \frac 1n \log \#\DDD_n.
\]
When $\DDD = \LLL(X)$, we write $P(X,\ph) = P(\LLL(X),\ph)$.
Let $\MMM_\sigma(X)$ denote the set of $\sigma$-invariant Borel probability measures on $X$.  The \emph{variational principle} \cite[Theorem 9.10]{pW82} says that
\[
P(X,\ph) = \sup \bigg\{ h_\mu(\sigma) + \int\ph\,d\mu \mid \mu\in \MMM_\sigma(X) \bigg\}.
\]
A measure 
achieving this supremum is called an \emph{equilibrium state}.  

Write $I(\ph) = \{\int\ph\,d\mu : \mu\in \MMM_\sigma(X)\}$.  Following \cite{IRRL12}, we call a potential function \emph{hyperbolic} if it satisfies \eqref{eqn:Pbig}; that is, if $P(X,\ph) > \sup 
I$.  Given $\eps>0$, there is $n\in \NN$ such that $\frac 1n S_n\ph(x) < \sup I + \eps$ for all $x\in X$; consequently, $\ph$ is hyperbolic if and only if there is $n\in \NN$ such that
\begin{equation}\label{eqn:hyperbolic}
P(X,\ph) > \sup_{x\in X} \frac 1n S_n \ph(x).
\end{equation}
Equivalently, one may observe that $\ph$ and $\frac 1n S_n\ph(x)$ are cohomologous,\footnote{Put $\xi(x) = \frac 1n\sum_{k=0}^{n-1} (n-k)\ph(\sigma_k x)$, then $\xi(x) - \xi(\sigma x) = \frac 1n S_n\ph(x) - \ph(x)$.} and so $\ph$ is hyperbolic if and only if there is a potential $\psi$ cohomologous to $\ph$ such that
\begin{equation}\label{eqn:hyperbolic-2}
P(X,\ph) = P(X,\psi) > \sup_{x\in X} \psi(x).
\end{equation}


\subsection{Specification, decompositions, and uniqueness}

Following the definition in \cite{CTY,spec-towers}, say that $\GGG \subset \LLL$ has \emph{specification} if there is $\tau>0$ such that for every $v,w\in \GGG$ there exists $u\in \LLL$ with length $|u|\leq \tau$ such that $vuw\in \GGG$.  This is a version of a condition that appeared in \cite{CTa,CTb} and generalises the classical specification property of Bowen \cite{rB74}, which corresponds roughly to this definition with $\GGG = \LLL$.

If $\GGG$ has specification with $\tau=0$, then we have $vw\in \GGG$ whenever $v,w\in \GGG$, and in this case we say that $\GGG$ has the \emph{free concatenation property}.  

When $\LLL$ has specification, it was proved by Bertrand \cite{aB88} that $\LLL$ contains a \emph{sychronising word}; that is, a word $s\in \LLL$ with the property that if $vs\in \LLL$ and $sw\in \LLL$, then $vsw\in \LLL$.  In this case the collection $\{sw : sws\in \LLL\}$ has the free concatenation property.  The following generalisation of this fact was proved in \cite[Proposition 7.3 and \S7.1.2]{spec-towers}.

\begin{proposition}\label{prop:get-free}
If $\GGG\subset \LLL$ has specification, then there is a collection $\FFF\subset \LLL$ and a number $N \in \NN$ such that 
\begin{enumerate}
\item $\FFF$ has the free concatenation property, and
\item given any $w\in \GGG$, there are $u,v\in \LLL$ with $|u|,|v| \leq N$ and $uwv\in \FFF$. 
\end{enumerate}
\end{proposition}

See \cite{spec-towers} for a more explicit description of the collection $\FFF$; all we will need are the properties listed above.  Writing $d = \gcd\{ |v| : v\in \FFF \}$, it follows from the free concatenation property that we can choose $N\in \NN$ large enough that $\FFF_n\neq\emptyset$ whenever $n\geq N$ is a multiple of $d$.  Thus Proposition \ref{prop:get-free} has the following consequence.

\begin{corollary}\label{cor:get-free}
Given $\GGG,\FFF$ as in Proposition \ref{prop:get-free} and $d$ as in the previous paragraph, there is $N\in \NN$ such that given any $w\in \GGG$ and any $n\geq |w| + 2N$ that is a multiple of $d$, there are $u,v\in \LLL$ with $|u|< N$, $uwv\in \FFF$, and $|uwv| = n$.
\end{corollary}
\begin{proof}
Let $N_0$ be given by Proposition \ref{prop:get-free}, and $N_1$ by the previous paragraph; then choose $N$ large enough that $N>N_0$ and $N-2N_0 \geq N_1$.  Given any $w\in \GGG$, Proposition \ref{prop:get-free} gives $u,v'\in \LLL$ with $|u|, |v'| \leq N_0 < N$ such that $uwv' \in \FFF$.  Let $n \geq |w| + 2N$ be a multiple of $d$.  By definition, $|uwv'|$ is a multiple of $d$, and thus $n-|uwv'|$ is also a multiple of $d$.  Moreover,
\[
n - |uwv'| \geq (|w| + N) - |w| - 2N_0 \geq N_1,
\]
so there is $v''\in \FFF$ with $|v''|=n-|uwv'|$, hence $uwv'v''\in \FFF$ and $|uwv'v''| = n$.
\end{proof}


If $\GGG$ is `large enough', then specification for $\GGG$ can be used to deduce uniqueness of the equilibrium state.  More precisely, a \emph{decomposition} of $\LLL$ is a choice of $\CCC^p, \GGG, \CCC^s \subset \LLL$ such that for every $w\in \LLL$ there are $u^p\in \CCC^p$, $v\in \GGG$, and $u^s\in \CCC^s$ with $w = u^p v u^s$.

\begin{theorem}[\cite{spec-towers}, Theorem 1.1]\label{thm:unique}
Suppose that $\GGG$ has specification and is closed under intersections and unions in the following sense: if $u,v,w\in \LLL$ are such that $uvw\in \LLL$, $uv\in \GGG$, and $vw\in \GGG$, then we have $v, uvw\in \GGG$.  Let $\ph$ be a H\"older potential and $\CCC^p \GGG \CCC^s$ a decomposition of $\LLL$ with $P(\CCC^p \cup \CCC^s,\ph) < P(\ph)$.  Then $\ph$ has a unique equilibrium state $\mu$, and $\mu$ has exponential decay of correlations (up to a finite period) and satisfies the central limit theorem for H\"older observables.
\end{theorem}

One can also use the results of \cite{CTb} to deduce uniqueness (but not the statistical properties) under extremely similar hypotheses.

\begin{remark}
For $\beta$-shifts and their factors, one can find a decomposition with $h(\CCC^p \cup \CCC^s) = 0$, and then the pressure gap condition in Theorem \ref{thm:unique} can be verified by proving hyperbolicity of the potential function, since an easy argument shows that $P(\DDD,\ph) \leq h(\DDD) + \sup_\mu \int\ph\,d\mu$ for every $\DDD \subset \LLL$.
\end{remark}

\subsection{Hamming approachability and asymptotic estimates}

Given a function $g\colon \NN\to \NN$, we say that $\LLL$ is \emph{$g$-Hamming approachable} by $\GGG \subset \LLL$ if there is $n_0\in \NN$ such that for every $n\geq n_0$ and $w\in \LLL_n$, there is $v\in \GGG_n$ with
\begin{equation}\label{eqn:dH}
\dH(v,w) :=
\# \{1\leq i\leq |w| :  v_i\neq w_i \} \leq g(|w|).
\end{equation}
This follows \cite[Definition 2.10]{CTY}, with the difference that we include the function $g$ in the notation, and will ultimately require that $g$ be sublogarithmic, not just sublinear.  We assume without loss of generality that $g$ is nondecreasing.

We will also need to use the fact that for any $k\leq m\in \NN$ and any $w\in \LLL_m$, we have
\begin{equation}\label{eqn:Ham-ball}
\#\{v\in \LLL_m : \dH(v,w) \leq k \} \leq \binom{m}{k} (\#A)^k.
\end{equation}
This becomes more useful with an estimate for $\binom{m}{k}$.  Recall from Stirling's formula that $\log(n!) = n\log n - n + O(\log n)$, 
and thus
\begin{align*}
\log\binom{m}{k} &= (m\log m - m) - (k\log k - k) \\
&\qquad - ((m-k)\log(m-k) - (m-k)) + O(\log m) \\
&= k\log\frac{m}{k} + (m-k)\log\frac{m}{m-k} + O(\log m).
\end{align*}
Writing $h(t) = -t\log t - (1-t)\log(1-t)$ for the bipartite entropy function, this gives
\begin{equation}\label{eqn:binom}
\log\binom{m}{k} = h\Big(\frac{k}{m}\Big) m + O(\log m),
\end{equation}
and so there is a constant $Q$ such that \eqref{eqn:Ham-ball} gives
\begin{equation}\label{eqn:Ham-ball-2}
\#\{v\in \LLL_m : \dH(v,w) \leq k \} \leq 
e^{m h(k/m)} m^Q (\#A)^k.
\end{equation}

\begin{lemma}\label{lem:Ham-diam}
Suppose $\DDD \subset \LLL$ has $h(\DDD)>0$,
and let $\beta>0$ be small enough that $h(\beta) + \beta\log(\#A) < h(\DDD)$.
Then for every $N\in \NN$ there are arbitrarily large $m\in \NN$ with the following property:
given any $w_1,\dots, w_N \in \DDD_m$, there is $v\in \DDD_m$ with $\dH(v,w_i) > \beta m$ for all $1\leq i\leq N$.
\end{lemma}
\begin{proof}
Choose $\eta,\xi>0$ such that $h(\beta) + \beta\log(\#A) + \xi < \eta < h(\DDD)$.  
Given $m\in \NN$ and $w_1,\dots, w_N\in \DDD_m$, \eqref{eqn:Ham-ball-2} gives
\[
\# \bigcup_{i=1}^N\{v\in \LLL_m : \dH(v,w_i) \leq \beta m \} \leq Ne^{mh(\beta)} m^Q (\#A)^{\beta m}
< N m^Q e^{(\eta - \xi)m}.
\]
The right-hand side is
$<\#\DDD_m$ whenever $N m^Q < e^{m\xi}$ and $\#\DDD_m \geq e^{m\eta}$; this happens infinitely often.
\end{proof}

\subsection{Coded systems}

Given a finite alphabet $A$ and a collection of words $G \subset A^*$, write $G^*$ for the set of all finite concatenations of words in $G$.  The \emph{coded shift} generated by $G$ is the subshift $X$ over the alphabet $A$ whose language consists of all subwords of elements of $G^*$.  We refer to $G$ as a \emph{generating set} for $X$.  The generating set is said to be \emph{uniquely decipherable} if whenever $u^1 u^2 \cdots u^m = v^1 v^2 \cdots v^n$ with $u^i, v^j \in G$, we have $m=n$ and $u^j = v^j$ for all $j$ \cite[Definition 8.1.21]{LM95}.

\begin{theorem}\cite[Theorem 1.8]{spec-towers}\label{thm:coded}
Let $X$ be a coded shift on a finite alphabet and $\ph$ a H\"older potential on $X$.  If $X$ has a uniquely decipherable generating set $G$ such that $\DDD = \DDD(G) := \{w\in \LLL : w$ is a subword of some $g\in G\}$ satisfies $P(\DDD,\ph) < P(\ph)$, then $\ph$ has a unique equlibrium state $\mu$, and $\mu$ has exponential decay of correlations (up to a finite period) and satisfies the central limit theorem for H\"older observables.
\end{theorem}

\section{Proof of Theorem \ref{thm:hyperbolic}}

In \S\ref{sec:prelim} we establish some preliminary results that are needed in order to describe precisely (in \S\ref{sec:generate}) the mechanism by which we generate entropy.

\subsection{Preliminaries for the proof}\label{sec:prelim}

We start with the following consequence of Corollary \ref{cor:get-free}.

\begin{lemma}\label{lem:get-free}
	Under the hypotheses of Theorem \ref{thm:hyperbolic}, there are $N\in \NN$ and $\FFF\subset \LLL$ with the free concatenation property such that writing $d = \gcd \{ |v| : v\in \FFF\}$, the following is true: for every $w\in \LLL$ such that $|w|\geq 2N$ and $|w|$ is a multiple of $d$, there is some $w'\in \FFF$ such that $|w|=|w'|$ and
	\begin{equation}\label{eqn:dHww'}
	\dH(w_{[1,|w|-i]}, w'_{(i,|w'|]}) \leq g(|w|) + 2N \text{ for some } 0\leq i\leq N-1.
	\end{equation}
\end{lemma}
\begin{proof}
Let $\FFF$ be as in Proposition \ref{prop:get-free} and $N$ as in Corollary \ref{cor:get-free}.  Then $x=w_{[1,|w|-2N]}$ has $y\in \GGG_{|w|-2N}$ such that $\dH(x,y) \leq g(|w|-2N) \leq g(|w|)$, where we use the fact that $g$ is nondecreasing.
Corollary \ref{cor:get-free} gives $u,v\in \LLL$ such that $|u| < N$, $uyv\in \FFF$ and $|uyv| = |w|$.  Let $w' = uyv$ and $i = |u|$; then writing $w=xzz'$ where $|z'| = i$, we have
\begin{align*}
\dH(w_{[1,|w|-i]}, w'_{(i,|w'|]}) &= \dH(xz,yv) = \dH(x,y) + \dH(z,v) \\
&\leq g(|w|) + |z| \leq g(|w|) + 2N.\qedhere
\end{align*}
\end{proof}

Consider the map  $\LLL_n \to \FFF_n$ given by $w\mapsto w'$ as in Lemma \ref{lem:get-free}.  By \eqref{eqn:Ham-ball-2}, the multiplicity of this map is at most $N e^{nh\big(\frac{g(n)+2N}{n-N}\big)} n^Q (\#A)^{g(n)+2N}$.  Writing $c_n$ for this quantity we observe that $\#\FFF_n \geq (\#\LLL_n) / c_n$ whenever $n$ is a multiple of $d$, and that $\lim_{n\to\infty} \frac 1n \log c_n = 0$, so $h(\FFF) = h(\LLL) = \htop(X) > 0$. 
Thus we can take $\beta>0$ small enough that $h(\beta) + \beta\log(\#A) < h(\FFF)$, and fix some $m\geq \max(3N, n_0)$ such that the conclusion of Lemma \ref{lem:Ham-diam} holds, where $n_0$ is as in the paragraph preceding \eqref{eqn:dH}.  Note that $m$ must be a multiple of $d = \gcd\{|v|:v\in \FFF\}$.

Now we fix several more parameters that will be used in the proof.  First we will find $V>0$ that controls $|\Phi(v)-\Phi(w)|$ in terms of $\dH(v,w)$; then we will choose $\gamma>0$ small relative to $m,V$; then we choose a large $L>0$ that helps us control $\sum_i g(n_i)$; and finally we will choose $\delta>0$ small enough that a certain entropy estimate later on is positive.

Let $\alpha>0$ be the H\"older exponent of $\ph$, and write $|\ph|_\alpha = \sup_{x\neq y} \frac{|\ph(x)-\ph(y)|}{d(x,y)^\alpha}$.  Then for every $n\in \NN$, $w\in \LLL_n$, and $x,y\in [w]$, we have
\[
|S_n\ph(x) - S_n\ph(y)| \leq \sum_{k=0}^{n-1} |\ph(\sigma^kx) - \ph(\sigma^ky)|
\leq \sum_{k=0}^{n-1} |\ph|_\alpha 2^{-(n-k)\alpha}
< \frac{|\ph|_\alpha}{1-2^{-\alpha}}.
\]
In particular, writing $V := |\ph|_\alpha(1-2^{-\alpha})^{-1}$, we have
\begin{equation}\label{eqn:Bowen}
|S_n\ph(x) - \Phi(w)| \leq V
\text{ for all } n\in \NN, w\in \LLL_n, \text{ and } x\in [w].
\end{equation}
This has the corollary that for every $v,w\in \LLL$ with $|v|=|w|$, we have
\begin{equation}\label{eqn:Bowen-2}
|\Phi(v) - \Phi(w)| \leq V\dH(v,w).
\end{equation}

\begin{lemma}\label{lem:gamma-L}
For every $\gamma>0$ there is $L>0$ such that for every $n_1,\dots, n_\ell\in \NN$ we have
\begin{equation}\label{eqn:sum-g}
\sum_{i=1}^\ell g(n_i) \leq \ell \Big( L + \gamma \log \frac{\sum n_i}\ell \Big).
\end{equation}
\end{lemma}
\begin{proof}
Since $g(n)/\log n \to 0$, there exists $K\in \NN$ such that
\begin{equation}\label{eqn:K}
g(n) < \gamma \log(n) \text{ for all } n>K.
\end{equation}
Let $L := \max\{g(n) : 1\leq n\leq K\}$.  Then we have the following estimate: given any $n>K$, $\ell\in \NN$, and
$n_1, \dots, n_{\ell}\in \NN$ such that $\sum_{i=1}^{\ell} n_i = n$, 
we have
\begin{equation}\label{eqn:sumgni}
\begin{aligned}
\sum_{i=1}^{\ell} g(n_i)
&\leq \sum_{\{ i : n_i \leq K\}} g(n_i) + \sum_{\{ i : n_i > K\}} g(n_i) \\
&\leq L \#\{i : n_i \leq K\} + \sum_{\{i : n_i > K\}} \gamma \log n_i \\
&\leq L \ell + \gamma \ell \log(n/\ell)
= \ell (L + \gamma\log (n/\ell)).
\end{aligned}
\end{equation}
The last inequality uses convexity; the function $(x_1, \dots, x_{\ell})\mapsto \sum_i \log x_i$ is maximized (subject to the constraint $\sum x_i = n$) when $x_1 = \cdots = x_{\ell} = n/\ell$, for which values we have $\sum_i \log x_i = \ell \log(n/\ell)$.
\end{proof}

For the duration of the proof, we fix $0 < \gamma < (16m^2 V)^{-1}$, and let $L$ be given by Lemma \ref{lem:gamma-L}.  Without loss of generality, we assume that $L\geq 2m$.
Finally, with $V,\beta,m,\gamma,L$ fixed, we choose $\delta>0$ small enough that
\begin{equation}\label{eqn:delta}
\frac{|\log\delta|}{8 m^2}
> 2\log \Big( \frac{2L + \gamma |\log\delta|}{\beta m} \Big) + 4VL.
\end{equation}

\subsection{Construction of nearby words}\label{sec:generate}

To prove hyperbolicity of $\ph$ it suffices to show that for every $x\in X$, we have
$P(\ph) > \ulim_{n\to\infty} \frac 1n S_n\ph(x)$.  To this end, we take $w\in \LLL$ to be a (sufficiently long) word, and estimate $\Lambda_{|w|}(\LLL,\ph)$ in terms of $e^{\Phi(w)}$.

Let $m\in \NN$ be as above.  Given $n \gg m$ with $(2m)|n$, fix  $k_n \in [\delta n, 2\delta n] \cap \NN$, and let
\[
\mathcal{J}_n = \{ \vn = (n_1,\dots, n_{k_n}) : \textstyle \sum n_i = n \text{ and } (2m)|n_i \text{ for all } i\}.
\]
Given $\vn \in \mathcal{J}_n$, let $N_j = n_1 + n_2 + \cdots + n_{j-1}$ be the partial sums.  For a fixed $w\in \LLL_n$, we will associate to each $\vn\in \mathcal{J}_n$ a word $\psi(\vn) \in \LLL_n$ such that 
\begin{enumerate}
\item $\psi(\vn)$ is Hamming-close to $w$ on the intervals $(N_i,N_{i+1}-m]$;
\item $\psi(\vn)$ is Hamming-far from $w$ on the intervals $(N_i - m, N_i]$.
\end{enumerate}
This will allow us to decipher $\vn$ from $\psi(\vn)$ up to some (controllable) error; that is, we will be able to control the multiplicity of the map $\psi\colon \mathcal{J}_n \to \LLL_n$.  Moreover, each $\psi(\vn)$ will have ergodic sum $\Phi(\psi(\vn))$ that is close to $\Phi(w)$.  These two facts, together with an estimate on $\#\mathcal{J}_n$, will give us the desired lower bound on $\Lambda_n(\LLL,\ph)$.

Let us make this more precise.  Given $\vn$, we have $n_i \geq 2m \geq m+2N$ for all $i$, and so applying Lemma \ref{lem:get-free} to $w_{(N_i,N_{i+1} - m]} \in \LLL_{n_i - m}$ gives $v^i\in \FFF_{n_i - m}$ such that
\begin{equation}\label{eqn:NN'}
\dH(w_{(N_i,N_{i+1} - m - a_i]}, v^i_{(a_i,n_i - m]}) \leq g(n_i) + 2N \text{ for some } 0\leq a_i < N.
\end{equation}
Consequently, we have
\begin{equation}\label{eqn:viw}
\dH(v^i, w_{(N_i - a_i, N_{i+1} - m - a_i])} \leq g(n_i) + 3N \leq g(n_i) + m.
\end{equation}
Moreover, by Lemma \ref{lem:Ham-diam} there are words $s^i \in \FFF_m$ such that
\begin{equation}\label{eqn:si}
\dH(s^i, w_{(N_i - m - a, N_i - a]}) \geq \beta m \text{ for all } 1\leq a \leq N.
\end{equation}
Now we can define the map $\psi =\psi_w \colon \mathcal{J}_n \to \LLL_n$ by 
\begin{equation}\label{eqn:psi}
\psi(\vn) = v^1s^1v^2s^2 \cdots v^{k_n} s^{k_n}.
\end{equation}
Summing over all $\vn\in \mathcal{J}_n$ gives
\[
\log\Lambda_n(\LLL,\ph)
\geq \Phi(w) + \log\#\mathcal{J}_n - \max_{\vn \in \mathcal{J}_n} |\Phi(\psi(\vn)) - \Phi(w)| - \max_{u\in \LLL_n} \#\psi^{-1}(u).
\]
If we divide both sides by $n$, send $n\to \infty$, and write
\begin{align*}
h_\JJJ &:= \llim_{n\to\infty} \frac 1n \log \#\JJJ_n, \\
\Delta_\Phi &:= \ulim_{n\to\infty} \frac 1n \max_{w\in \LLL_n}\max_{\vn \in \mathcal{J}_n} |\Phi(\psi_w(\vn)) - \Phi(w)|, \\
h_\psi &:= \ulim_{n\to\infty} \frac 1n \max_{w\in \LLL_n} \max_{u\in \LLL_n} \#\psi_w^{-1}(u),
\end{align*}
we get
\begin{equation}\label{eqn:main-P}
P(\ph) \geq \sup I + h_\JJJ - \Delta_\Phi - h_\psi,
\end{equation}
where we recall that
\[
I = \Big\{ \int\ph\,d\mu : \mu \in \MMM_\sigma(X)\Big\}
= \Big[ \inf_{x\in X} \llim_{n\to\infty} \frac 1n S_n\ph(x),
\sup_{x\in X} \ulim_{n\to\infty} \frac 1n S_n\ph(x) \Big].
\]
To complete the proof of Theorem \ref{thm:hyperbolic}, it suffices to show that $h_\JJJ > \Delta_\Phi + h_\psi$, which we do in the next section.

\subsection{Estimates on errors and entropy}

\subsubsection{Entropy gained from $\JJJ$}

Using \eqref{eqn:binom} and the definition of $\JJJ_n$, we have
\[
\log \#\mathcal{J}_n = \log\binom{\frac n{2m}}{k_n}
\geq h\Big( \frac{\delta}{2m}\Big) \frac n{2m} + O(\log n),
\]
and thus
\begin{equation}\label{eqn:hJ}
h_\JJJ \geq \frac \delta{4m^2} \Big|\log \frac \delta{2m} \Big|
\geq \frac{\delta}{4m^2} |\log\delta|.
\end{equation}

\subsubsection{Errors in ergodic sums}

Given any $w\in \LLL_n$ and $\vn\in \JJJ_n$, with $v^i$ as in the definition of $\psi$ we see from \eqref{eqn:Bowen-2} and \eqref{eqn:NN'} that
\[
|\Phi(w_{(N_i,N_{i+1}-m]}) - \Phi(v^i)| \leq (g(n_i) + 3N)V \leq (g(n_i) + m)V,
\]
and hence $|\Phi(w_{(N_i, N_{i+1}]} - \Phi(v^i s^i)| \leq (g(n_i) + 2m)V$.
Summing over all $i$ and using Lemma \ref{lem:gamma-L} gives
\[
|\Phi(\psi(\vn)) - \Phi(w)| \leq \sum_{i=1}^{k_n} (g(n_i) + 2m) V
\leq k_n(L + 2m + \gamma \log(n/k_n))V,
\]
and since $L\geq 2m$ we get
\begin{equation}\label{eqn:PhiPhi}
\max_{w\in \LLL_n} \max_{\vn\in \JJJ_n} |\Phi(\psi(\vn)) - \Phi(w)|
\leq k_n(2L + \gamma\log(n/k_n))V.
\end{equation}
Dividing by $n$ and using $k_n \in [\delta n, 2\delta n]$ gives
\begin{equation}\label{eqn:DeltaPhi}
\Delta_\Phi \leq 2\delta V(2L + \gamma |\log\delta|).
\end{equation}

\subsubsection{Multiplicity of $\psi$}

Given $u\in \LLL_n$, let
\begin{multline*}
R_u= \{ j\in [1,n] : m|j \text{ and } \dH(u_{[j,j+m)}, w_{[j-a, j+m-a)}) \geq \beta m \\ \text{ for all } 0\leq a< N\}.
\end{multline*}
It follows from \eqref{eqn:si} that $\{N_i\}_{i=1}^{k_n} \subset R_{\psi(\vn)}$ for all $\vn\in \mathcal{J}_n$.  Moreover, given $\vn\in \mathcal{J}_n$ we see from \eqref{eqn:viw} that $u = \psi(\vn)$ has
\begin{equation}\label{eqn:ai}
\sum_{j=N_i/m}^{(N_{i+1}/m) - 1}
\dH(u_{(jm,(j+1)m]},w_{(jm-a_i,(j+1)m-a_i]}) \leq g(n_i) + 2m
\end{equation}
for every $1\leq i\leq n_k$, and summing over $i$ gives
\begin{equation}\label{eqn:betamRu}
\begin{aligned}
\beta m \cdot \#R_u
&\leq \sum_{j=1}^{n/m} \min_{0\leq a <  N} \dH(u_{[jm,jm+m)},w_{[jm-a,jm+m-a)}) \\
&\leq \sum_{i=1}^{k_n} (g(n_i) + 2m )
\leq k_n (2L + \gamma |\log\delta| ),
\end{aligned}
\end{equation}
where the last inequality again uses Lemma \ref{lem:gamma-L} and the inequalities $L\geq 2m$, $k_n \geq \delta n$.  Thus we have
\[
\#R_u \leq k_n \cdot \frac{2L + \gamma|\log\delta|}{\beta m},
\]
and since $\vn\in \mathcal{J}_n$ is determined by a choice of $k_n$ elements from $R_u$, we conclude from \eqref{eqn:binom} that
\[
\log \#\psi^{-1}(u) 
\leq h\Big( \frac{\beta m}{2L + \gamma|\log \delta|} \Big) \frac{2\delta n}{\beta m}(2L+\gamma|\log\delta|) + 
O(\log n),
\]
and so
\begin{equation}\label{eqn:hpsi}
\begin{aligned}
h_\psi &\leq \frac{\beta m}{2L + \gamma|\log \delta|} \log \Big(\frac{2L+\gamma|\log\delta|}{\beta m}\Big) \frac{2\delta}{\beta m}(2L + \gamma|\log\delta|) \\
&=
2\delta \log\Big(\frac{2L+\gamma|\log\delta|}{\beta m}\Big)
\end{aligned}
\end{equation}

\subsubsection{Completion of the proof}

Combining \eqref{eqn:hJ}, \eqref{eqn:DeltaPhi}, and \eqref{eqn:hpsi}, we get
\[
\frac{h_\JJJ - \Delta_\Phi - h_\psi}{\delta} \geq
\frac{|\log \delta|}{4m^2}
- 4VL - 2V\gamma|\log \delta| - 2\log\Big(\frac{2L+\gamma|\log\delta|}{\beta m}\Big).
\]
Since we chose $\gamma$ to be smaller than $(16m^2 V)^{-1}$, we have
\[
\frac{|\log \delta|}{8m^2} - 2V\gamma|\log\delta| > 0,
\]
and thus
\[
\frac{h_\JJJ - \Delta_\Phi - h_\psi}{\delta} >
\frac{|\log \delta|}{8m^2} - 4VL - 2\log\Big(\frac{2L+\gamma|\log\delta|}{\beta m}\Big).
\]
The right-hand side is positive by our choice of $\delta$ in \eqref{eqn:delta}, and we conclude that $h_\JJJ > \Delta_\Phi + h_\psi$.  By \eqref{eqn:main-P}, this gives $P(\ph) > \sup I$, which completes the proof of Theorem \ref{thm:hyperbolic}.

\section{Proof of Theorem \ref{thm:context-free}}

Now we consider the shift space $X$ described in Theorem \ref{thm:context-free}.  Write $\LLL$ for the language of $X$ and $f\colon \NN\to \NN$ for the function used to define $G = \{0^a 1^b : a,b\geq f(a+b)\}$.  
Recall that $\ph = -\one_{[1]}$.  Before we prove the five statements listed in the theorem, we demonstrate that $P(t\ph)$ is nonnegative and nonincreasing.
Let $\delta_0$ be the $\delta$-measure on the fixed point $0 \in X$.  Then 
for every $t\in \RR$ we have
$P(t\ph) \geq h_{\delta_0}(\sigma) + t \int \ph\,d\delta_0 = t\ph(0) = 0$.
Since $\ph\leq 0$ it follows from basic properties of pressure that whenever $s < t$, we have $P(t\ph) = P(s\ph + (t-s)\ph) \leq P(s\ph + (t-s)0) = P(s\ph)$, so the pressure function is nonincreasing.

\subsection{Hamming approachability}\label{sec:hamming-approachability}

Let $n_1$ be such that $f(n) \leq n/2$ for all $n\geq n_1$; in particular, for all $n\geq n_1$ there are $a,b\geq f(n)$ such that $a+b=n$, and thus $0^a1^b \in G$.  We need the following lemma.

\begin{lemma}\label{lem:u01v}
Given $n\geq n_1$ and $w\in \LLL_n$, suppose that $w$ can be written as $w=u0^a1^bv$ for some $u,v\in \LLL$ with $|u|,|v|\leq n_1$ and $a,b\geq 0$.  (Note that $u,v$ are allowed to be empty.)
Then there is $\tilde w\in G$ such that $\dH(w,\tilde w) \leq n_1 + \max(f(n),n_1)$.
\end{lemma}
\begin{proof}
If $n_1 + a < f(n)$, then $\tilde w = 0^{f(n)} 1^{n-f(n)} \in G$ satisfies
\begin{multline*}
\dH(w, \tilde w) \leq 
\dH(w_{[1,f(n)]}, 0^{f(n)}) + \dH(w_{(f(n),n]}, 1^{n-f(n)}) \\
 = \dH(u0^a 1^{f(n) - a - |u|}, 0^{f(n)}) + \dH(1^{n-f(n)-|v|}v,1^{n-f(n)}) \leq f(n) + n_1.
\end{multline*}
Similarly, if $n_1 + a > n-f(n)$, then $\tilde w = 0^{n-f(n)} 1^{f(n)} \in G$ satisfies
\[
\dH(w, \tilde{w}) \leq n_1 + f(n).
\]
Finally, if $f(n) \leq n_1 + a \leq n-f(n)$, then $\tilde{w} = 0^{n_1+a}1^{n-n_1-a}$ satisfies $\dH(w,\tilde w) \leq 2n_1$.
\end{proof}

Now given any $w\in \LLL$ with $|w| \geq 2n_1$, there are integers $0 = \ell_0 <\ell_1 < \cdots < \ell_m = n$ such that
\begin{gather*}
w_{(\ell_{i-1},\ell_i]} = 0^{a_i} 1^{b_i}
\text{ for all } 1\leq i\leq m, \\
a_i, b_i \geq f(a_i+b_i) \text{ for all } 1 < i < m, \\
a_1,b_1,a_m,b_m \geq 0.
\end{gather*}
Choose $0\leq j \leq k \leq m$ such that
\[
n_1 \in (\ell_{j-1},\ell_j] \text{ and }
n-n_1 \in (\ell_{k-1},\ell_k].
\]
If $j=k$ then $w$ has the form required for Lemma \ref{lem:u01v}, and thus there is $\tilde w \in G$ such that $\dH(w,\tilde w) \leq n_1 + \max(f(n),n_1)$.  If $j<k$, then we can write $w = w^p w^c w^s$, where
\[
w^p := w_{(0,\ell_{j+1}]},
\quad w^c := w_{(\ell_{j+1},\ell_k]},
\quad w^s := w_{(\ell_k,n]}.
\]
Note that $w^c \in \FFF$, and $w^p,w^s$ both have the form required for Lemma \ref{lem:u01v}, so taking $\tilde w^p$ and $\tilde w^s$ as given by that lemma, we have $\tilde w^p w^c \tilde w^s \in G^*$ and
\[
\dH(w,\tilde w^p w^c \tilde w^s)
\leq \dH(w^p,\tilde w^p) + \dH(w^s,\tilde w^s)
\leq 2n_1 + 2\max(f(n),n_1).
\]
This proves the first item in Theorem \ref{thm:context-free}.

\subsection{Hyperbolicity when $P(t\ph)>0$}

Let $I_t = \{\int t\ph\,d\mu : \mu\in \MMM_\sigma(X)\}$.  
The second statement in Theorem \ref{thm:context-free} is equivalent to the claim that when $t\geq 0$, we have $P(t\ph) > \sup I_t$ if and only if $t<t_0$, where $t_0$ is the first root of Bowen's equation \eqref{eqn:t0}.  Since $t\mapsto P(t\ph)$ is nonincreasing, we see that $t<t_0$ if and only if $P(t\ph) > 0$.  On the other hand, since $\int \ph\,\delta_1 = -1\leq \ph\leq 0 = \int\ph\,\delta_0$, we have $I_t = [-t,0]$ for all $t\geq 0$, and so $\sup I_t = 0$, which proves the desired equivalence.

\subsection{Unique equilibrium state when $t<t_0$}

To deduce uniqueness of the equilibrium state for $t\ph$ when $0\leq t<t_0$, we apply Theorem \ref{thm:coded}.  (Positive entropy of the equilibrium state will then follow since $t\ph$ is hyperbolic.)  The shift $X$ is coded with generating set $G = \{0^a 1^b : a,b\geq f(a+b)\}$.  This is uniquely decipherable because if $w = u^1 u^2 \cdots u^m$ with $u^i\in G$, then we can recover $u^1$ from $w$ as the longest initial segment of the form $0^a 1^b$ with $a,b\geq 1$, then $u^2$ from the remainder of $w$ by the same procedure, and so on.
Moreover, the set
\[
\DDD = \DDD(G) := \{w\in \LLL : w\text{ is a subword of some } g\in G \}
\]
is easily seen to satisfy $\DDD \subset \{0^a 1^b : a,b\geq 0\}$, and hence $\#\DDD_n \leq n+1$, so $h(\DDD) = 0$.  We conclude that
\[
P(\DDD,t\ph) \leq h(\DDD) + \sup I_t = \sup I_t \text{ for all } t,
\]
and since we showed that $t\ph$ is hyperbolic whenever $0\leq t < t_0$, we conclude that $P(\DDD,t\ph) < P(t\ph)$ for this range of $t$, and so we can apply Theorem \ref{thm:coded}.

\subsection{Only the delta measure past $t_0$}

Since $t\mapsto P(t\ph)$ is nonincreasing and nonnegative, we have $P(t\ph) = P(t_0\ph) = 0$ for all $t\geq t_0$.  Thus $\delta_0$ is an equilibrium state for all $t\geq t_0$.  When $t>t_0$, we observe that every other $\mu\in \MMM_\sigma(X)$ has $\mu[1]>0$ and hence $\int \ph\,d\mu < 0$, so
\begin{multline*}
h_\mu(\sigma) + \int t\ph\,d\mu = h_\mu(\sigma) + \int t_0\ph\,d\mu + \int (t-t_0) \ph\,d\mu  \\
\leq P(t_0 \ph) + (t-t_0)\int\ph\,d\mu < 0,
\end{multline*}
which shows that $\delta_0$ is the unique equilibrium state on this range of $t$.

\subsection{Bowen's equation has a root if and only if $\sum \gamma^{f(n)} < \infty$}

For the final statement in Theorem \ref{thm:context-free}, we fix $t>0$ and study the power series
\[
F(x):=\sum_{n=1}^{\infty}\Lambda_n(G,t\ph)x^n 
\quad\text{and}\quad
H(x):=1+\sum_{n=1}^{\infty}\Lambda_n(G^*,t\ph)x^n.
\]

\begin{proposition}\label{prop:tfae}
For the shift space in Theorem \ref{thm:context-free} and $t>0$, the following are equivalent.
\begin{enumerate}[label=\textup{(\alph{*})}]
\item\label{a} $P(t\ph) = 0$.
\item\label{b} The power series $H(x)$ 
converges for every $0\leq x < 1$.
\item\label{c} The power series $F(x)$ converges for every $0\leq x < 1$, with $F(x) < 1$.
\item\label{d} The power series $F(x)$ converges for $x=1$, with $F(1) \leq 1$.
\end{enumerate}
\end{proposition}
\begin{proof}
\ref{a}$\Leftrightarrow$\ref{b}.
Consider the power series $A(x) = \sum_{n=0}^\infty \Lambda_n(X,t\ph) x^n$ (here $\Lambda_0(X,t\ph) = 1$).
Since $\lim \sqrt[n]{\Lambda_n(X,t\ph)} = e^{P(t\ph)}$, the root test tells us that the radius of convergence of $A(x)$ is $e^{-P(t\ph)} \leq 1$.  In particular, $P(t\ph) = 0$ if and only if $A(x)$ converges for every $0\leq x < 1$, so to prove the first equivalence it suffices to show that the power series $A(x)$ and $H(x)$ converge for the same values of $x\in [0,1)$.
To this end, consider the sets of words
\[
\mathcal{P} = \{0^a1^b : a < f(a+b)\}\text{ and }
\mathcal{S} = \{0^a1^b : b < f(a+b)\}.
\]
Every $w\in \LLL$ admits a unique decomposition as $w = u^p v u^s$ for some $u^p\in \mathcal{P}$, $v\in G^*$, and $u^s\in \mathcal{S}$, and since 
$\Phi(u^pvu^s) = \Phi(u^p) + \Phi(v) + \Phi(u^s)$, we have
\begin{equation}\label{eqn:abc}
\sum_{n=0}^N \Lambda_n(X,t\ph) x^n = \sum_{\substack{a,b,c\geq 0 \\ a+b+c\leq N}} \Lambda_a(\mathcal{P},t\ph) x^a
\Lambda_b(G^*,t\ph) x^b
\Lambda_c(\mathcal{S},t\ph) x^c.
\end{equation}
Consider the power series associated to $\mathcal{P}$ and $\mathcal{S}$:
\[
\Cp(x):=1+\sum_{n=1}^{\infty}
\Lambda_n(\mathcal{P},\ph)x^n
\text{ and }
\Cs(x):=1+\sum_{n=1}^{\infty}
\Lambda_n(\mathcal{S},\ph)x^n.
\]
Write $H_N,A_N, \Cp_N, \Cs_N$ for the partial sums (over $n\leq N$) of the respective power series; then \eqref{eqn:abc} gives
\begin{equation}\label{eqn:N3N}
\Cp_N(x)H_N(x) \Cs_N(x) \leq A_{3N}(x) \leq \Cp_{3N}(x) H_{3N}(x) \Cs_{3N}(x).
\end{equation}
We claim that $\Cp(x)$ and $\Cs(x)$ both converge for all $0\leq x < 1$.  For $\Cs(x)$ we have 
\[
\Cs(x)=1+\sum_{n=1}^{\infty}\left(\sum_{k=0}^{f(n)-1}e^{-tk}\right)x^n
= 1+\sum_{n=1}^{\infty}\left(\frac{1-e^{-tf(n)}}{1-e^{-t}}\right)x^n,
\] 
which has radius of convergence $x=1$ since the 
coefficients lie in the interval $(0,1]$.
Similarly for $\Cp(x)$, we have
\[ 
\Cp(x)=1+\sum_{n=1}^{\infty}\left(\sum_{k=0}^{f(n)-1}e^{-t(n-k)}\right)x^n
=1+\sum_{n=1}^{\infty}\left(\frac{e^{-t(n-f(n))}-e^{-tn}}{e^t-1}\right)x^n,
\] 
and since $1\leq f(n)\leq n/2$ for all sufficiently large $n$, the coefficients converge to $0$ and the radius of convergence of $\Cp(x)$ is greater than or equal to $x=1$.  Thus $\Cp(x)$ and $\Cs(x)$ both converge for all $0\leq x < 1$, and it follows from \eqref{eqn:N3N} that for every such $x$, $H(x)$ converges if and only if $A(x)$ converges.  This proves the equivalence of \ref{a} and \ref{b}.

\medskip\noindent
\ref{b}$\Leftrightarrow$\ref{c}.
Since $X$ is uniquely decipherable we have 
\[ 
\Lambda_n(G^*,t\ph)=\sum_{j=1}^n\sum_{n_1+\cdots+n_j=n}\prod_{i=1}^j\Lambda_{n_i}(G,t\ph).\] 
It follows that whenever $|F(x)|<1$ we have 
\begin{equation}\label{eq:useful} 
H(x)=1+\sum_{k=1}^{\infty}F(x)^k=\frac{1}{1-F(x)} 
\end{equation} 
and if $0\leq x<1$ is such that $F(x)\geq1$, then $H(x)$ does not converge. 

\medskip\noindent
\ref{c}$\Leftrightarrow$\ref{d}. Suppose $F(1)$ converges.  Then $F(x)$ converges for all $|x|<1$ by standard facts on power series, and since all the coefficients are nonnegative (and not all of them vanish), the function $F$ is strictly increasing on $[0,1]$, so $0\leq F(x) < F(1)$ for all $x\in [0,1)$, which proves \ref{d}$\Rightarrow$\ref{c}.

Now we prove  \ref{c}$\Rightarrow$\ref{d}.
Suppose that for all $0\leq x<1$ we have $F(x) < 1$.  Then the partial sums $F_N(x)$ also satisfy $F_N(x)<1$ for all $x\in [0,1)$ and $N\in \NN$, since the coefficients are nonnegative.  By continuity we get $F_N(1) \leq 1$ for all $N\in \NN$, and thus $F(1) \leq 1$.
\end{proof}

By Proposition \ref{prop:tfae}, in order to complete the proof of Theorem \ref{thm:context-free}\ref{root} it suffices to show that there is $t>0$ with $F(1)\leq 1$ if and only if there is $\gamma>0$ such that $\sum_n \gamma^{f(n)} < \infty$.   
Observe that
\begin{equation}\label{eqn:LG}
\Lambda_n(G,t\ph)=\sum_{k=f(n)}^{n-f(n)}e^{-tk}=\frac{e^{-t(f(n)-1)}-e^{-t(n-f(n))}}{e^t-1}
\end{equation}
whenever $f(n) \leq n/2$, and $\Lambda_n(G,t\ph) = 0$ otherwise.
Since $f(n) \leq n/2$ for all sufficiently large $n$, we have
\[
\sum\frac{e^{-t(n-f(n))}}{e^t-1} < \infty,
\]
implying that $F(1)<\infty$ if and only if $\sum_{n=1}^\infty e^{-t(f(n)-1)} / (e^t-1) < \infty$.  In particular, if $F(1) \leq 1$ then $\sum \gamma^{f(n)} < \infty$ for $\gamma = e^{-t}$.  

For the converse direction, suppose that $\gamma>0$  is such that $\sum\gamma^{f(n)}<\infty$.  Then for all $t \geq  -\log\gamma$, \eqref{eqn:LG} gives
\[
\sum_{n=1}^\infty \Lambda_n(G,t\ph) \leq \sum_{n=1}^\infty \frac{e^{-t(f(n)-1)}}{e^t-1}
\leq \sum_{n=1}^\infty \frac{\gamma^{f(n)-1} }{e^t-1} 
\leq \frac 1{\gamma(e^t-1)} \sum_{n=1}^\infty \gamma^{f(n)}.
\]
By taking $t$ sufficiently large, the right-hand side can be made $\leq 1$, so for this value of $t$ we have $F(1) \leq 1$, which completes the proof of Theorem \ref{thm:context-free}.

\bibliographystyle{amsalpha}
\bibliography{hyperbolic-es}

\providecommand{\bysame}{\leavevmode\hbox to3em{\hrulefill}\thinspace}
\providecommand{\MR}{\relax\ifhmode\unskip\space\fi MR }
\providecommand{\MRhref}[2]{%
  \href{http://www.ams.org/mathscinet-getitem?mr=#1}{#2}
}
\providecommand{\href}[2]{#2}
\begin{thebibliography}{IRRL12}

\bibitem[Ber88]{aB88}
Anne Bertrand, \emph{Specification, synchronisation, average length}, Coding
  theory and applications ({C}achan, 1986), Lecture Notes in Comput. Sci., vol.
  311, Springer, Berlin, 1988, pp.~86--95. \MR{960710}

\bibitem[BK90]{BK90}
V.~Baladi and G.~Keller, \emph{Zeta functions and transfer operators for
  piecewise monotone transformations}, Comm. Math. Phys. \textbf{127} (1990),
  no.~3, 459--477. \MR{1040891}

\bibitem[Bow75]{rB74}
Rufus Bowen, \emph{Some systems with unique equilibrium states}, Math. Systems
  Theory \textbf{8} (1974/75), no.~3, 193--202. \MR{0399413}

\bibitem[Buz01]{jB01}
J{\'e}r{\^o}me Buzzi, \emph{Thermodynamical formalism for piecewise invertible
  maps: absolutely continuous invariant measures as equilibrium states}, Smooth
  ergodic theory and its applications ({S}eattle, {WA}, 1999), Proc. Sympos.
  Pure Math., vol.~69, Amer. Math. Soc., Providence, RI, 2001, pp.~749--783.
  \MR{1858553}

\bibitem[Buz04]{jB04}
\bysame, \emph{Entropy of equilibrium measures of continuous piecewise
  monotonic maps}, Stoch. Dyn. \textbf{4} (2004), no.~1, 84--94. \MR{2069369}

\bibitem[CFT19]{CFT}
Vaughn Climenhaga, Todd Fisher, and Daniel~J. Thompson, \emph{Equilibrium
  states for {M}a\~n\'e diffeomorphisms}, Ergodic Theory Dynam. Systems (2019),
  to appear,
  \href{http://arxiv.org/abs/1703.05722}{\color{red}arXiv:1703.05722}.

\bibitem[Cli18]{spec-towers}
Vaughn Climenhaga, \emph{Specification and towers in shift spaces}, Comm. Math.
  Phys. \textbf{364} (2018), no.~2, 441--504.

\bibitem[Con]{sC}
Scott Conrad, \emph{{A coded shift with a H\"older potential that is not
  hyperbolic}}, preprint.

\bibitem[CT12]{CTa}
Vaughn Climenhaga and Daniel~J. Thompson, \emph{Intrinsic ergodicity beyond
  specification: {$\beta$}-shifts, {$S$}-gap shifts, and their factors}, Israel
  J. Math. \textbf{192} (2012), no.~2, 785--817. \MR{3009742}

\bibitem[CT13]{CTb}
\bysame, \emph{Equilibrium states beyond specification and the {B}owen
  property}, J. Lond. Math. Soc. (2) \textbf{87} (2013), no.~2, 401--427.
  \MR{3046278}

\bibitem[CTY17]{CTY}
Vaughn Climenhaga, Daniel~J. Thompson, and Kenichiro Yamamoto, \emph{Large
  deviations for systems with non-uniform structure}, Trans. Amer. Math. Soc.
  \textbf{369} (2017), no.~6, 4167--4192. \MR{3624405}

\bibitem[DKU90]{DKU90}
Manfred Denker, Gerhard Keller, and Mariusz Urba{\'n}ski, \emph{On the
  uniqueness of equilibrium states for piecewise monotone mappings}, Studia
  Math. \textbf{97} (1990), no.~1, 27--36. \MR{1074766}

\bibitem[HK82]{HK82}
Franz Hofbauer and Gerhard Keller, \emph{Equilibrium states for piecewise
  monotonic transformations}, Ergodic Theory Dynam. Systems \textbf{2} (1982),
  no.~1, 23--43. \MR{684242}

\bibitem[Hof79]{fH79}
Franz Hofbauer, \emph{On intrinsic ergodicity of piecewise monotonic
  transformations with positive entropy}, Israel J. Math. \textbf{34} (1979),
  no.~3, 213--237 (1980). \MR{570882}

\bibitem[IRRL12]{IRRL12}
Irene Inoquio-Renteria and Juan Rivera-Letelier, \emph{A characterization of
  hyperbolic potentials of rational maps}, Bull. Braz. Math. Soc. (N.S.)
  \textbf{43} (2012), no.~1, 99--127. \MR{2909925}

\bibitem[Kel84]{gK84}
Gerhard Keller, \emph{On the rate of convergence to equilibrium in
  one-dimensional systems}, Comm. Math. Phys. \textbf{96} (1984), no.~2,
  181--193. \MR{768254}

\bibitem[LM95]{LM95}
Douglas Lind and Brian Marcus, \emph{An introduction to symbolic dynamics and
  coding}, Cambridge University Press, Cambridge, 1995. \MR{1369092}

\bibitem[LRL14]{LRL14}
Huaibin Li and Juan Rivera-Letelier, \emph{Equilibrium states of weakly
  hyperbolic one-dimensional maps for {H}\"older potentials}, Comm. Math. Phys.
  \textbf{328} (2014), no.~1, 397--419. \MR{3196990}

\bibitem[PS07]{PS07}
C.-E. Pfister and W.~G. Sullivan, \emph{On the topological entropy of saturated
  sets}, Ergodic Theory Dynam. Systems \textbf{27} (2007), no.~3, 929--956.
  \MR{2322186}

\bibitem[Ryc83]{mR83}
Marek Rychlik, \emph{Bounded variation and invariant measures}, Studia Math.
  \textbf{76} (1983), no.~1, 69--80. \MR{728198}

\bibitem[Tho12]{dT12}
Daniel~J. Thompson, \emph{Irregular sets, the {$\beta$}-transformation and the
  almost specification property}, Trans. Amer. Math. Soc. \textbf{364} (2012),
  no.~10, 5395--5414. \MR{2931333}

\bibitem[Wal82]{pW82}
Peter Walters, \emph{An introduction to ergodic theory}, Graduate Texts in
  Mathematics, vol.~79, Springer-Verlag, New York-Berlin, 1982. \MR{648108}

\end{thebibliography}

\end{document}